\documentclass[12pt,a4paper]{amsart}
\usepackage{amsfonts}
\usepackage{amsthm}
\usepackage{amsmath}
\usepackage{amscd}
\usepackage[latin2]{inputenc}

\def\({\left(}       \def\){\right)}

\theoremstyle{plain}
\newtheorem{Th}{Theorem}[section]
\newtheorem{Lem}[Th]{Lemma}

\newtheorem{Prop}[Th]{Proposition}

 \theoremstyle{definition}

\newtheorem{?}[Th]{Problem}

\begin{document}

\title{ Hausdorff operators on  Bergman spaces of the upper half plane}

\author{Georgios Stylogiannis}
\maketitle
\begin{abstract}
  In this paper we study  Hausdorff operators  on the    Bergman spaces $A^{p}(\mathbb{U})$ of the upper half plane.
\end{abstract}

\section{introduction}
Given a  $\sigma$-finite positive Borel measure
$\mu$ on $ (0,\infty)$, the associated  Hausdorff operator $\mathcal{H}_\mu$,
 for suitable functions $f$ is given by
\begin{equation}\label{Df Hausd 2}
\mathcal{H}_{\mu}(f)(z):=\int_{0}^{\infty}\frac{1}{t}f\left(\frac{z}{t}\right)\,d\mu(t),\quad z\in \mathbb{U}
\end{equation}
where $\mathbb{U}=\{z\in\mathbb{C}: \mbox{Im}\,z>0\}$ is the upper half plane.
Its formal adjoint, the quasi-Hausdorff operator $\mathcal{H}^{*}_{\mu}$ in the case of real Hardy spaces $H^p(\mathbb{R})$ is
\begin{equation}\label{Df Hausd 1}
\mathcal{H}_{\mu}^{*}(f)(z):=\int_{0}^{\infty}f(tz)\,d\mu(t).
\end{equation}

Moreover for appropriate  functions $f$ and measures $\mu$ they satisfy the fundamental identity:

\begin{equation}\label{Df Hausd 3}
\widehat{\mathcal{H}_{\mu}(f)}(x)=\int_{0}^{\infty}\widehat{f}(tx)\, d\mu(t)=
\mathcal{H}_{\mu}^{*}(\widehat{f})(x),\quad x\in\mathbb{R},
\end{equation}
where $\widehat{f}$ denotes the Fourier transform of $f$.

The  theory of Hausdorff summability of Fourier series
started with the paper of Hausdorff \cite{Ha21} in 1921. Much later
 Hausdorff summability of power series of analytic functions  was considered in \cite{Si87} and \cite{Si90} on composition operators and the Ces\'{a}ro means in  Hardy $H^p$ spaces. General Hausdorff matrices were
considered in \cite{GaSi01} and \cite{GaPa06}. In \cite{GaPa06} the authors studied Hausdorff matrices on a large class of analytic function spaces such as Hardy spaces, Bergman spaces, BMOA, Bloch  etc. They  characterized those    Hausdorff matrices which induce bounded operators on these spaces.\\

Results on Hausdorff operators on spaces of analytic functions were extended in the Fourier
transform setting on the real line, starting with \cite{LiMo00} and \cite{Ka01}. There are
many classical operators in analysis which are special cases of the Hausdorff operator
if one chooses suitable measures $\mu$ such as the classical Hardy operator, its
adjoint operator, the Ces\'{a}ro type operators and the Riemann-Liouville fractional integral
operator. See the survey article \cite{Li013} and the references there in. In recent years, there
is an increasing interest on the study of boundedness of the Hausdorff operator on the
real Hardy spaces and Lebesque spaces (see for example \cite{An03}, \cite{BaGo19}, \cite{FaLi14},  \cite{LiMo01} and \cite{HuKyQu18}).

Motivated by the paper of Hung et al. \cite{HuKyQu18} we  describe the measures $\mu$ that will induce bounded operators on the Bergman spaces $A^{p}(\mathbb{U})$ of the upper half-plane. Next Theorem summarizes the main results  (see Theorems \ref{B Haus Berg} and \ref{Norm Haus Berg} ):
\begin{Th}
Let $1\leq p< \infty$ and $\mu$ be an $\sigma$-finite positive measure on $(0,\infty)$. The Hausdorff operator $\mathcal{H}_{\mu}$ is bounded on $A^{p}(\mathbb{U})$ if and only if
$$
\int_{0}^{\infty}\frac{1}{t^{1-\frac{2}{p}}}\,d\mu(t)<\infty.
$$
Moreover
$$
||\mathcal{H}_{\mu}||_{A^{p}(\mathbb{U})\to A^{p}(\mathbb{U})}=\int_{0}^{\infty}\frac{1}{t^{1-\frac{2}{p}}}\,d\mu(t).
$$

\end{Th}
\section{preliminaries}

To define  single-valued functions, the principal value of the argument it is chosen to
be in the interval $(-\pi, \pi]$.  For $1\leq p < \infty$, we denote by $L^{p}(dA)$
the Banach space of all measurable functions on $\mathbb{U}$ such that
$$
||f||_{L^{p}(dA_{a})}:=\left(\frac{1}{\pi}\int_{\mathbb{U}}|f|^p dA\right)^{1/p}<\infty,
$$
where $dA$ is the area measure.
The  Bergman space $A^{p}(\mathbb{U})$  consists  of all
holomorphic functions $f$ on $\mathbb{U}$ that belong to $L^{p}(dA)$.
Sub-harmonicity yields a constant $C> 0$ such that
\begin{equation}\label{Grouth in A(p,a)}
|f(z)|^{p} \leq \frac{C}{(\mbox{Im}(z))^{2}}||f||^{p}_{A^{p}(\mathbb{U})},\quad z\in \mathbb{U},
\end{equation}
for $f \in A^{p}(\mathbb{U})$ and
$$
\lim_{z\to \partial{\widehat{\mathbb{U}}}} (\mbox{Im}(z))^{2}|f(z)|^p = 0
$$
for functions $A^{p}(\mathbb{U})$, where $\widehat{\mathbb{U}} :=
\overline{\mathbb{U}}\cup\{\infty\}$ (see \cite{ChKoSm17}). In particular, this
shows that each point evaluation is a continuous linear functional on
$A^{p}(\mathbb{U})$.

The duality properties of Bergman spaces are well known in literature see\cite{Zh90} and \cite{BaBoMiMi16}.
It is proved that for $1 < p < \infty$, $\frac{1}{p}+\frac{1}{q}=1$, the dual space of the Bergman space $A^{p}(\mathbb{U})$ is
$(A^{p}(\mathbb{U}))^{*}\sim A^{q}(\mathbb{U})$ under the duality pairing,
$$
\langle f,g\rangle= \frac{1}{\pi}\int_{\mathbb{U}}f(z)\overline{g(z)}\,dA(z).
$$

\section{Main results}
In what follows, unless otherwise stated,  $\mu$ is a positive $\sigma$-finite
measure on $(0,\infty)$. We start by  giving a condition under which
$\mathcal{H}_{\mu}$ is well defined.

\begin{Lem}\label{Well defined lemma 1}
Let $1\leq p<\infty$ and $f\in A^{p}(\mathbb{U})$. If $
\int_{0}^{\infty}\frac{1}{t^{1-\frac{2}{p}}}\,d\mu(t)<\infty, $ then
$$
\mathcal{H}_{\mu}(f)(z)=\int_{0}^{\infty}\frac{1}{t}f\left(\frac{z}{t}\right)\,d\mu(t)
$$
 is a well defined holomorphic function on $\mathbb{U}$.
\end{Lem}
\begin{proof}
 For $f\in A^{p}(\mathbb{U})$ using (\ref{Grouth in A(p,a)}) we have

\begin{align*}
|\mathcal{H}_{\mu}(f)(z)|&\leq \int_{0}^{\infty}\frac{1}{t}\left|f\left(\frac{z}{t}\right)\right|\,d\mu(t)\\
&\leq  C\,\frac{||f||_{A^{p}(\mathbb{U})}}{\mbox{Im}(z)^{\frac{2}{p}}}\int_{0}^{\infty}\frac{1}{t^{1-\frac{2}{p}}}\,d\mu(t)<\infty.
\end{align*}
Thus $\mathcal{H}_{\mu}f$ is well defined, and is given by an absolutely convergent integral, so it is holomorphic.
\end{proof}

\begin{Lem}\label{Grouth f(e)}
Let $\lambda>0$  and $\delta>0$. If $g_{\lambda, \delta}(z)=|z+\delta i|^{-\frac{2+\lambda}{p}}$,  then
$$
\left(\frac{1}{2}\right)^{2+\lambda} \cdot \frac{1}{\lambda \delta^{\lambda}}\leq ||g_{\lambda,\delta}||_{L^{p}(dA)}^{p}\leq2^{\frac{2+\lambda}{2}} \cdot\frac{1}{\lambda\delta^{\lambda}}
$$
\end{Lem}
\begin{proof}
Using polar coordinates for the integral over $\mathbb{U}$ we find
\begin{align*}
||g_{\lambda,\delta}||_{L^{p}(dA)}^{p}
&=\frac{1}{\pi}\int_{\mathbb{U}}\left|\frac{1}{z+ \delta i}\right|^{2+\lambda}dA(z)\\
&=\frac{1}{\pi}\int_{0}^{\pi}\int_{0}^{\infty}\left(\frac{1}{r^{2}+\delta^{2}+2r\delta\sin(\theta)}\right)^{\frac{2+\lambda}{2}}r\,drd\theta.\\
\end{align*}
Denote by $I$ the last double integral. Then
\begin{align*}
I&\leq   \int_{0}^{\infty}\left(\frac{1}{r^{2}+\delta^{2}}\right)^{\frac{2+\lambda}{2}}rdr\\
&\leq 2^{\frac{2+\lambda}{2}} \int_{0}^{\infty}\left(\frac{1}{r+\delta}\right)^{2+\lambda}(r+\delta)dr\\
&=2^{\frac{2+\lambda}{2}}\frac{1}{\lambda\delta^{\lambda}}.
\end{align*}
On the other hand,
\begin{align*}
I&\geq   \int_{0}^{\infty}\left(\frac{1}{r+\delta}\right)^{2+\lambda}rdr\\
&\geq  \int_{\delta}^{\infty}\left(\frac{1}{r+\delta}\right)^{2+\lambda}rdr\\
&\geq \left(\frac{1}{2}\right)^{2+\lambda} \int_{\delta}^{\infty}\left(\frac{1}{r}\right)^{2+\lambda}rdr\\
&=\left(\frac{1}{2}\right)^{2+\lambda} \frac{1}{\lambda \delta^{\lambda}},
\end{align*}
and the assertion follows.
\end{proof}

\subsection{Test functions} We now  consider the test functions which are defined as follows. Let
$z=x+iy\in \mathbb{U}$ and
$$
\varphi_{\varepsilon}(z)=\frac{\overline{z+\varepsilon i}}{|z+\varepsilon i|}
=\frac{x-i(y+\varepsilon)}{\sqrt{x^2+(y+\varepsilon)^2}}
$$
and
$$
f_{\varepsilon}(z)=\frac{1}{(z+ \varepsilon i)^{\frac{2}{p}+\varepsilon}},
$$
with $\varepsilon>0$  small enough. Note that, $|f_{\varepsilon}|\equiv
g_{p\varepsilon,\varepsilon}$ with respect to the notation of Lemma \ref{Grouth
f(e)}, and that $\varphi_{\varepsilon}(z)$  lies on the unit circle with
$-\pi<\arg(\varphi_{\varepsilon}(z))< 0$, and the following identity
holds
\begin{align*}
f_{\varepsilon}(z)=\varphi_{\varepsilon}(z)^{\frac{2}{p}+\varepsilon}|f_{\varepsilon}(z)|.
\end{align*}

Let $a,b \in (-\pi,\pi]$ and set  $A_{[a,b]}= \{ z\in\mathbb{U}: a\leq \arg(z)\leq b\}$, with obvious modifications   in the case of $A_{(a,b)}, A_{(a,b]}$ and $A_{[a,b)}$.

\begin{Lem}\label{Grouth f(e) 1}
The following holds:\\

\noindent $(i)$ If $2<p<\infty$ and $\frac{2}{p}+\varepsilon\leq1$, then
\begin{equation*}
|\mbox{Re}\,f_{\varepsilon}(z)|\geq |\mbox{Re}\,\varphi_{\varepsilon}(z)|
|f_{\varepsilon}(z)|
\end{equation*}
for every $z\in A_{(0,\frac{\pi}{2}]}$.\\

\noindent $(ii)$ If $1<p\leq 2$ and $1<\frac{2}{p}+\varepsilon<2$, then
\begin{equation*}
|\mbox{Im}\,f_{\varepsilon}(z)|> C(p)|\mbox{Im}\,\varphi_{\varepsilon}(z)|
|f_{\varepsilon}(z)|
\end{equation*}
for every $z\in A_{[\frac{\pi}{4},\frac{\pi}{2}]}$.\\

\noindent $(ii)$ If $p=1$, $0<\theta_0<\frac{\pi}{16}$ and $(2+\varepsilon)(\frac{\pi}{2}+\theta_0)<\frac{5\pi}{4}$, then
\begin{equation*}
|\mbox{Re}\,f_{\varepsilon}(z)|> |\mbox{Re}\,\varphi_{\varepsilon}(z)|
|f_{\varepsilon}(z)|
\end{equation*}
for every $z\in A_{[\frac{\pi}{2},\frac{\pi}{2}+\theta_{0}]}$
\end{Lem}
\begin{proof}

Taking real and imaginary parts we have
\begin{align*}
\mbox{Re}\,f_{\varepsilon}(z)=|f_{\varepsilon}(z)|\mbox{Re}\,\varphi_{\varepsilon}(z)^{\frac{2}{p}+\varepsilon}=
|f_{\varepsilon}(z)|\cos((\frac{2}{p}+\varepsilon)\theta)
\end{align*}
and
\begin{align*}
\mbox{Im}\,f_{\varepsilon}(z)=|f_{\varepsilon}(z)|\mbox{Im}\,\varphi_{\varepsilon}(z)^{\frac{2}{p}+\varepsilon}=
|f_{\varepsilon}(z)|\sin((\frac{2}{p}+\varepsilon)\theta),
\end{align*}
where $\theta:=\theta(z,\varepsilon)=\arg\varphi_{\varepsilon}(z)$. \\

\noindent $(i)$: It easy is see that
\begin{align*}
|\mbox{Re}\,f_{\varepsilon}(z)|&=|f_{\varepsilon}(z)|\cos((\frac{2}{p}+\varepsilon)\theta)\\
&\geq |f_{\varepsilon}(z)|\cos(\theta)=|\mbox{Re}\,\varphi_{\varepsilon}(z)|
|f_{\varepsilon}(z)|.
\end{align*}

\noindent $(ii)$:  Let $a=a(p)>0$ such that $1<\frac{2}{p}+\varepsilon<a<2$. Since $z\in A_{[\frac{\pi}{4}, \frac{\pi}{2}]}$  simple geometric arguments imply  that $\theta \in [-\frac{\pi}{2}, -\frac{\pi}{4})$. Moreover $-\pi<- \frac{a\pi}{2}<(\frac{2}{p}+\varepsilon) \theta < -\frac{\pi}{4}$. This implies that
$$
\min\{\sin(\frac{a\pi}{2}),\frac{\sqrt{2}}{2}\}<\left|\frac{\sin((\frac{2}{p}+\varepsilon) \theta)}{\sin(\theta)}\right|< \sqrt{2}
$$
for every $z\in  A_{[\frac{\pi}{4}, \frac{\pi}{2}]}$. We calculate
\begin{align*}
&|\mbox{Im}\,f_{\varepsilon}(z)|=|f_{\varepsilon}(z)||\sin((\frac{2}{p}+\varepsilon)\theta)|> \min\{\sin(\frac{a\pi}{2}),\frac{\sqrt{2}}{2}\} |f_{\varepsilon}(z)||\sin(\theta)|\\
&= \min\{\sin(\frac{a\pi}{2}),\frac{\sqrt{2}}{2}\}|\mbox{Im}\,\varphi_{\varepsilon}(z)|
|f_{\varepsilon}(z)|.
\end{align*}
This proves $(ii)$ with $C(p)=\min\{\sin(\frac{a(p)\pi}{2}),\frac{\sqrt{2}}{2}\}$.\\

\noindent $(iii)$: Since $z\in A_{[\frac{\pi}{2},\frac{\pi}{2}+\theta_{0}]}$  we have that $\theta\in (-\frac{\pi}{2}-\theta_{0}, -\frac{\pi}{2}]$. Thus
$$
-\frac{5\pi}{4}<-(\frac{\pi}{2}+\theta_0)(2+\varepsilon)< (2+\varepsilon)\theta\leq -\frac{\pi}{2}(2+\varepsilon)<-\pi.
$$
This implies that $|\cos((2+\varepsilon)\theta)|>|\cos(\theta)|$ and therefore
\begin{align*}
|\mbox{Re}\,f_{\varepsilon}(z)|&>|\mbox{Re}\,\varphi_{\varepsilon}(z)|
|f_{\varepsilon}(z)|.
\end{align*}
\end{proof}

\subsection{Growth estimates}

Let $a,b \in (-\pi,\pi]$ and set
$$
S_{[a,b]}= \{ z\in\mathbb{U}: a\leq \arg(z)\leq b, \, |z|\geq1\},
$$
be a truncated sector with obvious modifications in the case of $S_{(a,b)}, S_{(a,b]}$ and $S_{[a,b)}$.
 Since $\mu$ is positive
 $$\mbox{Re}\, \mathcal{H}_{\mu}(f_{\varepsilon})=\mathcal{H}_{\mu}(\mbox{Re}\,f_{\varepsilon})\quad \mbox{and} \quad
 \mbox{Im}\, \mathcal{H}_{\mu}(f_{\varepsilon})=\mathcal{H}_{\mu}(\mbox{Im}\,f_{\varepsilon}).
 $$
 Note that if $\mbox{Re}\,f_{\varepsilon}$ or $\mbox{Im}\,f_{\varepsilon}$  have constant  sign on some sector $A$, then
 $$
|\mathcal{H}_{\mu}(\mbox{Re}\,f_{\varepsilon})(z)|=\mathcal{H}_{\mu}(|\mbox{Re}\,f_{\varepsilon}|)(z) \quad \mbox{and} \quad|\mathcal{H}_{\mu}(\mbox{Im}\,f_{\varepsilon})(z)|=\mathcal{H}_{\mu}(|\mbox{Im}\,f_{\varepsilon}|)(z)
$$
for every $z\in A$.
\begin{Lem}\label{Grouth Hf(e) 1}
Let  $1\leq p<\infty$ and suppose that $\mathcal{H}_{\mu}$ is bounded on
$A^{p}(\mathbb{U})$.  Then there are $\varepsilon(p)$ and $k(p)$ positive constants such that
$$
||\mathcal{H}_{\mu}(f_{\varepsilon})||_{A^{p}(\mathbb{U})}^{p}\geq k(p)\left(\int_{0}^{\frac{1}{\varepsilon}} \frac{1}{t^{1-\frac{2}{p}-\varepsilon}}d\mu(t)\right)^{p}\frac{1}{p\varepsilon},
$$
for every  $\varepsilon$ in $(0,\varepsilon(p)]$.
\end{Lem}
\begin{proof}
We will consider three cases for the range of $p$. Note that if $z$ is in a truncated sector $S$ then $z/t$ belongs to the corresponding  sector  $A$ for every $t>0$.  \\

\noindent \textbf{Case I}. Let $2<p<\infty$ and $\varepsilon(p)$ such that $\frac{2}{p}+\varepsilon(p)<1$.  Then for every $\varepsilon$ in $(0,\varepsilon(p)]$
\begin{align*}
||\mathcal{H}_{\mu}(f_{\varepsilon})||_{A^{p}(\mathbb{U})}^{p}&\geq ||\mbox{Re}\,\mathcal{H}_{\mu}(f_{\varepsilon})||_{A^{p}(\mathbb{U})}^{p}
=||\mathcal{H}_{\mu}(\mbox{Re}\, f_{\varepsilon})||_{A^{p}(\mathbb{U})}^{p}\\
&\geq\frac{1}{\pi}\int_{S_{(0,\frac{\pi}{2}]}}\left|\int_{0}^{\infty}\frac{1}{t}\mbox{Re}\, f_{\varepsilon}(z/t)\,d\mu(t)\right|^{p} dA(z)\\
&= \frac{1}{\pi}\int_{S_{(0,\frac{\pi}{2}]}}\left(\int_{0}^{\infty}\frac{1}{t}|\mbox{Re}\, f_{\varepsilon}(z/t)|\,d\mu(t)\right)^{p} dA(z).\\
\end{align*}
Denote by $I$ the last integral on $S_{(0,\frac{\pi}{2}]}$. 
 By $(i)$ of Lemma \ref{Grouth f(e) 1} we have
\begin{align*}
&I\geq \frac{1}{\pi}\int_{S_{(0,\frac{\pi}{2}]}}\left(\int_{0}^{\frac{1}{\varepsilon}}\frac{1}{t}|\mbox{Re}\, \varphi_{\varepsilon}(z/t)|\,|f_{\varepsilon}(z/t)|\,d\mu(t)\right)^{p} dA(z)\\
&=\frac{1}{\pi}\int_{S_{(0,\frac{\pi}{2}]}}\left(\int_{0}^{\frac{1}{\varepsilon}} \left(\frac{1}{\sqrt{x^2+ (y+t\varepsilon)^2}}\right)^{\frac{2}{p}+\varepsilon+1}\frac{d\mu(t)}{t^{1-\frac{2}{p}-\varepsilon}}\right)^{p} x^{p}dxdy.\\
\end{align*}
Using polar coordinates and noting that $t\varepsilon <|z|$ for $|z|\geq1$ and $t\leq \varepsilon^{-1}$, we have
\begin{align*}
&I\geq\int_{1}^{\infty} \int_{0}^{\frac{\pi}{2}}\left(\int_{0}^{\frac{1}{\varepsilon}} \left(\frac{1}{\sqrt{r^2+ 2rt\varepsilon\sin(\theta)+t^{2}\varepsilon^2}}\right)^{\frac{2}{p}+\varepsilon+1}\frac{d\mu(t)}{t^{1-\frac{2}{p}-\varepsilon}}\right)^{p}  r^{p+1}(\cos(\theta))^{p}\frac{d\theta}{\pi} dr\\
&\geq\int_{1}^{\infty} \int_{0}^{\frac{\pi}{2}}\left(\int_{0}^{\frac{1}{\varepsilon}} \left(\frac{1}{\sqrt{r^2+ 2rt\varepsilon+t^{2}\varepsilon^2}}\right)^{\frac{2}{p}+\varepsilon+1}\frac{d\mu(t)}{t^{1-\frac{2}{p}-\varepsilon}}\right)^{p}  r^{p+1}(\cos(\theta))^{p}\frac{d\theta}{\pi} dr\\
&\geq k(p)\left(\int_{0}^{\frac{1}{\varepsilon}} \frac{1}{t^{1-\frac{2}{p}-\varepsilon}}d\mu(t)\right)^{p}\int_{1}^{\infty}\frac{1}{r^{1+p\varepsilon}}dr\\
&=k(p)\left(\int_{0}^{\frac{1}{\varepsilon}} \frac{1}{t^{1-\frac{2}{p}-\varepsilon}}d\mu(t)\right)^{p}\frac{1}{p\varepsilon},
\end{align*}
where $k(p)=2^{-p(\varepsilon+1)}\int_{0}^{\frac{\pi}{2}}(\cos(\theta))^{p}\frac{d\theta}{4\pi}$.\\

\noindent \textbf{Case II}. Let $2<p<\infty$ and $\varepsilon(p)$ such that $1<\frac{2}{p}+\varepsilon(p)<2$.  Then for every $\varepsilon$ in $(0,\varepsilon(p)]$
\begin{align*}
||\mathcal{H}_{\mu}(f_{\varepsilon})||_{A^{p}(\mathbb{U})}^{p}&
\geq\frac{1}{\pi}\int_{S_{[\frac{\pi}{4},\frac{\pi}{2}]}}\left|\int_{0}^{\infty}\frac{1}{t}\mbox{Im}\, f_{\varepsilon}(z/t)\,d\mu(t)\right|^{p} dA(z)\\
&= \frac{1}{\pi}\int_{S_{[\frac{\pi}{4},\frac{\pi}{2}]}}\left(\int_{0}^{\infty}\frac{1}{t}|\mbox{Im}\, f_{\varepsilon}(z/t)|\,d\mu(t)\right)^{p} dA(z).\\
\end{align*}
Denote by $I$ the last integral on $S_{[\frac{\pi}{4},\frac{\pi}{2}]}$. 
By $(ii)$ of Lemma \ref{Grouth f(e) 1} we have
\begin{align*}
&I\geq \frac{C(p)}{\pi}\int_{S_{[\frac{\pi}{4},\frac{\pi}{2}]}}\left(\int_{0}^{\frac{1}{\varepsilon}}\frac{1}{t}|\mbox{Im}\, \varphi_{\varepsilon}(z/t)|\,|f_{\varepsilon}(z/t)|\,d\mu(t)\right)^{p} dA(z)\\
&\geq\frac{C(p)}{\pi}\int_{S_{[\frac{\pi}{4},\frac{\pi}{2}]}}\left(\int_{0}^{\frac{1}{\varepsilon}} \left(\frac{1}{\sqrt{x^2+ (y+t\varepsilon)^2}}\right)^{\frac{2}{p}+\varepsilon+1}\frac{d\mu(t)}{t^{1-\frac{2}{p}-\varepsilon}}\right)^{p} y^{p}dxdy.\\
\end{align*}
Using polar coordinates and working as is Case I, we arrive at the desired conclusion with constant $k(p)=C(p)2^{-p(\varepsilon+1)}\int_{\frac{\pi}{4}}^{\frac{\pi}{2}}(\sin(\theta))^{p}\frac{d\theta}{4\pi}$.\\

\noindent \textbf{Case III}. Let $p=1$ and $\theta_0$ as in Lemma \ref{Grouth f(e) 1}. Let $\varepsilon(1)$ such that $(2+\varepsilon(1))(\frac{\pi}{2}+\theta_{0})<\frac{5\pi}{4}$. Then for every $\varepsilon$ in $(0,\varepsilon(1)]$
\begin{align*}
||\mathcal{H}_{\mu}(f_{\varepsilon})||_{A^{1}(\mathbb{U})}&
\geq\frac{1}{\pi}\int_{S_{[\frac{\pi}{2},\frac{\pi}{2}+\theta_{0}]}}\left|\int_{0}^{\infty}\frac{1}{t}\mbox{Re}\, f_{\varepsilon}(z/t)\,d\mu(t)\right| dA(z)\\
&= \frac{1}{\pi}\int_{S_{[\frac{\pi}{2},\frac{\pi}{2}+\theta_{0}]}}\int_{0}^{\infty}\frac{1}{t}|\mbox{Re}\, f_{\varepsilon}(z/t)|\,d\mu(t)\, dA(z).\\
\end{align*}
Denote by $I$ the last integral on $S_{[\frac{\pi}{2},\frac{\pi}{2}+\theta_{0}]}$. 
 By $(iii)$ of Lemma \ref{Grouth f(e) 1} we have
\begin{align*}
&I\geq \frac{1}{\pi}\int_{S_{[\frac{\pi}{2},\frac{\pi}{2}+\theta_{0}]}}\int_{0}^{\frac{1}{\varepsilon}}\frac{1}{t}|\mbox{Re}\, \varphi_{\varepsilon}(z/t)|\,|f_{\varepsilon}(z/t)|\,d\mu(t)\, dA(z)\\
&=\frac{1}{\pi}\int_{S_{[\frac{\pi}{2},\frac{\pi}{2}+\theta_{0}]}}\int_{0}^{\frac{1}{\varepsilon}} \left(\frac{1}{\sqrt{x^2+ (y+t\varepsilon)^2}}\right)^{3+\varepsilon}\frac{d\mu(t)}{t^{-1-\varepsilon}} (-x)dxdy.\\
\end{align*}
Using polar coordinates and working as is Case I, we arrive at the desired conclusion with constant $k(1)=-2^{-(\varepsilon+1)}\int_{\frac{\pi}{2}}^{\frac{\pi}{2}+\theta_{0}}cos(\theta)\frac{d\theta}{4\pi}$.\\
\end{proof}

\begin{Th}\label{B Haus Berg}
Let  $1\leq p<\infty$. The operator $\mathcal{H}_{\mu}$ is bounded on
$A^{p}(\mathbb{U})$  if and only if
$$
\int_{0}^{\infty}\frac{1}{t^{1-\frac{2}{p}}}\,d\mu(t)<\infty.
$$
\end{Th}
\begin{proof}
Suppose that
$$
\int_{0}^{\infty}\frac{1}{t^{1-\frac{2}{p}}}\,d\mu(t)<\infty,
$$
then Lemma \ref{Well defined lemma 1} implies that $\mathcal{H}_{\mu}(f)$ is well defined and holomorphic in $\mathbb{U}$.
An easy computation involving the Minkowski inequality shows that for all
$1\leq p<\infty$
\begin{align*}
||\mathcal{H}_{\mu}(f)||_{A^{p}(\mathbb{U})}&=\left(\int_{\mathbb{U}}\left|\int_{0}^{\infty}\frac{1}{t}f\left(\frac{z}{t}\right)\,d\mu(t)\right|^{p}dA(z)\right)^{1/p}\\
&\leq ||f||_{A^{p}(\mathbb{U})}\int_{0}^{\infty}\frac{1}{t^{1-\frac{2}{p}}}\, d\mu(t)<\infty.
\end{align*}
Thus $\mathcal{H}_{\mu}$ is bounded on $A^{p}(\mathbb{U})$.\\

Conversely, suppose that $\mathcal{H}_{\mu}$ is bounded. Let $f_{\varepsilon}(z)=(z+ \varepsilon i)^{-(\frac{2}{p}+\varepsilon)}$ with $\varepsilon>0$ small enough. By Lemma \ref{Grouth f(e)}
\begin{equation}\label{norm f(e)}
||f_\varepsilon||^{p}_{A^{p}(\mathbb{U})}\sim \frac{1}{p\varepsilon \varepsilon^{p\varepsilon}}.
\end{equation}
Moreover Lemma \ref{Grouth Hf(e) 1} implies that there is a constant $k=k(p)>0$  such that

\begin{align*}
||f_\varepsilon||^{p}_{A^{p}(\mathbb{U})} ||\mathcal{H}_{\mu}||_{A^{p}(\mathbb{U})\to A^{p}(\mathbb{U})}&\geq ||\mathcal{H}_{\mu}(f_{\varepsilon})||_{A^{p}(\mathbb{U})}^{p}\\
&\geq k \left(\int_{0}^{\frac{1}{\varepsilon}} \frac{1}{t^{1-\frac{2}{p}-\varepsilon}}d\mu(t)\right)^{p}\frac{1}{p\varepsilon}.
\end{align*}
Thus by letting $\varepsilon\to 0$, we have in comparison to (\ref{norm f(e)})
 $$
\int_{0}^{\infty}\frac{1}{t^{1-\frac{2}{p}}}\,d\mu(t)<\infty.
$$
\end{proof}

In our way to computing the norm of $\mathcal{H}_{\mu}$ we will firstly  compute the norm of the truncated Hausdorff operator $\mathcal{H}_{\mu}^{\delta}$ given by :
 $$
\mathcal{H}_{\mu}^{\delta}(f)(z):=\int_{0}^{\infty}\frac{1}{t}f\left(\frac{z}{t}\right)X_{[\delta,1/\delta]}(t)\,d\mu(t)=
\int_{\delta}^{\frac{1}{\delta}}\frac{1}{t}f\left(\frac{z}{t}\right)\,d\mu(t).
$$

\begin{Prop}\label{Norm Haus Berg Truncate}
Let  $1\leq p<\infty$ and $0<\delta<1$. If
$$
\int_{0}^{\infty}\frac{1}{t^{1-\frac{2}{p}}}\,d\mu(t)<\infty,
$$
then $\mathcal{H}_{\mu}^{\delta}$ is bounded with
$$
||\mathcal{H}_{\mu}^{\delta}||_{A^{p}(\mathbb{U})\to A^{p}(\mathbb{U})}=\int_{\delta}^{\frac{1}{\delta}}\frac{1}{t^{1-\frac{2}{p}}}\,d\mu(t).
$$
\end{Prop}
\begin{proof}
 As in Theorem \ref{B Haus Berg} an application of Minkowski inequality gives
\begin{equation}\label{Cut Hausd}
||\mathcal{H}_{\mu}^{\delta}||_{A^{p}(\mathbb{U})\to A^{p}(\mathbb{U})}\leq \int_{\delta}^{\frac{1}{\delta}}\frac{1}{t^{1-\frac{2}{p}}}d\mu(t).
\end{equation}
Let $f_{\varepsilon}(z)=(z+i)^{-\frac{2}{p}-\varepsilon}$ with $\varepsilon>0$ small enough. We calculate
\begin{align*}
\mathcal{H}_{\mu}^{\delta}(f_{\varepsilon})(z)&-f_{\varepsilon}(z)\int_{\delta}^{\frac{1}{\delta}}\frac{1}{t^{1-\frac{2}{p}}}d\mu(t)\\
&=\int_{\delta}^{\frac{1}{\delta}}\frac{1}{t^{1-\frac{2}{p}}}\left(\varphi_{\varepsilon,z}(t)-\varphi_{\varepsilon,z}(1)\right)d\mu(t),
\end{align*}
where
$$
\varphi_{\varepsilon,z}(t)=\frac{t^{\varepsilon}}{(z+ti)^{\frac{2}{p}+\varepsilon}}.
$$
For any $t \in [\delta, 1/\delta]$,  calculus gives
\begin{align*}
\left|\varphi_{\varepsilon,z}(t)-\varphi_{\varepsilon,z}(1)\right|&\leq|t-1|\sup\{|\varphi_{\varepsilon,z}'(s)|:s\in [\delta, 1/\delta]\}\\
&\leq \frac{1}{\delta}\left(\frac{\varepsilon \delta^{\varepsilon-1}}{|z+i\delta|^{\frac{2}{p}+\varepsilon}}+
 \frac{(\frac{2}{p}+\varepsilon) (1/\delta)^{\varepsilon}}{|z+i\delta|^{\frac{2}{p}+\varepsilon+1}} \right)\\
 &=\varepsilon \delta ^{\varepsilon-2}g_{p\varepsilon,\delta}(z)+ (\frac{2}{p}+\varepsilon) (1/\delta)^{\varepsilon+1}g_{p(\varepsilon+1),\delta}(z).
\end{align*}
Where above we followed the notation of Lemma \ref{Grouth f(e)}.
Thus by an easy application of Minkowski inequality followed by the triangular inequality we have
\begin{align*}
||\mathcal{H}_{\mu}^{\delta}(f_{\varepsilon})(z)&-f_{\varepsilon}(z)\int_{\delta}^{\frac{1}{\delta}}\frac{1}{t^{1-\frac{2}{p}}}d\mu(t)||_{A^{p}(\mathbb{U})}\\
&\leq\int_{\delta}^{\frac{1}{\delta}}\frac{1}{t^{1-\frac{2}{p}}}||\left(\varphi_{\varepsilon,z}(t)-\varphi_{\varepsilon,z}(1)\right)||_{A^{p}(\mathbb{U})}d\mu(t)\\
&\leq   \int_{\delta}^{\frac{1}{\delta}}\frac{1}{t^{1-\frac{2}{p}}}\,d\mu(t)\left(\varepsilon \delta ^{\varepsilon-2}||g_{p\varepsilon,\delta}||_{L^{p}(dA)}+(\frac{2}{p}+\varepsilon) (1/\delta)^{\varepsilon+1}||g_{p(\varepsilon+1),\delta}||_{L^{p}(dA)}\right).
\end{align*}

This, together with Lemma \ref{Grouth f(e)} (recall that $|f_{\varepsilon}|=g_{p\varepsilon,\varepsilon}$), yields

\begin{align*}
&\frac{||\mathcal{H}_{\mu}^{\delta}(f_{\varepsilon})(z)-f_{\varepsilon}(z)\int_{\delta}^{\frac{1}{\delta}}\frac{1}{t^{1-\frac{2}{p}}}
d\mu(t)||_{A^{p}(\mathbb{U})}}{||f_{\varepsilon}||_{A^{p}(\mathbb{U})}}\\
&\leq  \int_{\delta}^{\frac{1}{\delta}}\frac{1}{t^{1-\frac{2}{p}}}\,d\mu(t)\times
\frac{\varepsilon \delta ^{\varepsilon-2}||g_{p\varepsilon,\delta}||_{L^{p}(dA)}+(\frac{2}{p}+\varepsilon) (1/\delta)^{\varepsilon+1}||g_{p(\varepsilon+1),\delta}||_{L^{p}(dA)}}{||f_{\varepsilon}||_{A^{p}(\mathbb{U})}}\to 0
\end{align*}
as $\varepsilon \to 0$. This and  (\ref{Cut Hausd}) imply that

$$
\int_{\delta}^{\frac{1}{\delta}}\frac{1}{t^{1-\frac{2}{p}}}d\mu(t)= ||\mathcal{H}_{\mu}^{\delta}||_{A^{p}(\mathbb{U})\to A^{p}(\mathbb{U})}.
$$
\end{proof}
\begin{Th}\label{Norm Haus Berg}
Let  $1\leq p<\infty$. If
$$
\int_{0}^{\infty}\frac{1}{t^{1-\frac{2}{p}}}\,d\mu(t)<\infty
$$
then
$$
||\mathcal{H}_{\mu}||_{A^{p}(\mathbb{U})\to A^{p}(\mathbb{U})}=\int_{0}^{\infty}\frac{1}{t^{1-\frac{2}{p}}}\,d\mu(t).
$$
\end{Th}

\begin{proof}
By Theorem \ref{B Haus Berg} we have that   
$$
||\mathcal{H}_{\mu}||_{A^{p}(\mathbb{U})\to A^{p}(\mathbb{U})}\leq \int_{0}^{\infty}\frac{1}{t^{1-\frac{2}{p}}}\, d\mu(t).
$$
 Minkowski inequality implies that
\begin{equation}\label{Diff Hausd}
||\mathcal{H}_{\mu}-\mathcal{H}_{\mu}^{\delta}||_{A^{p}(\mathbb{U})\to A^{p}(\mathbb{U})}\leq \int_{(0,\delta)\cup (\frac{1}{\delta},\infty)}\frac{1}{t^{1-\frac{2}{p}}}d\mu(t).
\end{equation}
By Proposition \ref{Norm Haus Berg Truncate}

$$
\int_{\delta}^{\frac{1}{\delta}}\frac{1}{t^{1-\frac{2}{p}}}d\mu(t)= ||\mathcal{H}_{\mu}^{\delta}||_{A^{p}(\mathbb{U})\to A^{p}(\mathbb{U})}.
$$
This, combined with (\ref{Diff Hausd}), allows us to conclude that
$$
||\mathcal{H}_{\mu}||_{A^{p}(\mathbb{U})\to A^{p}(\mathbb{U})}\geq \int_{0}^{\infty}\frac{1}{t^{1-\frac{2}{p}}}d\mu(t)-
2\int_{(0,\delta)\cup(1/\delta,\infty)}\frac{1}{t^{1-\frac{2}{p}}}d\mu(t)\to\int_{0}^{\infty}\frac{1}{t^{1-\frac{2}{p}}}d\mu(t)
$$
as $\delta\to 0$.  Hence,
$$
||\mathcal{H}_{\mu}||_{A^{p}(\mathbb{U})\to A^{p}(\mathbb{U})}=\int_{0}^{\infty}\frac{1}{t^{1-\frac{2}{p}}}d\mu(t).
$$
\end{proof}

\subsection{The quasi-Hausdorff operator}

Let $f,g\in A^{2}(\mathbb{U})$ and assume that $\mathcal{H}_{\mu}$ is bounded on $A^{2}(\mathbb{U})$. Thus
$$
\int_{0}^{\infty}\frac{1}{t^{1-\frac{2}{p}}}d\mu(t)<\infty.
$$
We have
\begin{align*}
\int_{\mathbb{U}}\int_{0}^{\infty}\frac{1}{t}\left|f\left(\frac{z}{t}\right)\right||g(z)|\, d\mu(t) dA(z)&\leq
\left(\int_{\mathbb{U}}\left(\int_{0}^{\infty}\frac{1}{t}\left|f\left(\frac{z}{t}\right)\right|\, d\mu(t)\right)^{2} dA(z)\right)^{1/2}||g||_{A^{2}(\mathbb{U})}\\
&\leq \left(\int_{0}^{\infty}\frac{1}{t^{1-\frac{2}{p}}}d\mu(t)\right)^{1/2} ||f||_{A^{2}(\mathbb{U})}||g||_{A^{2}(\mathbb{U})}<\infty,
\end{align*}
where we applied the Cauchy-Schwarz  and Minkowski inequalities. Therefore
\begin{align*}
\langle \mathcal{H}_{\mu}(f),g\rangle &=\int_{\mathbb{U}}\mathcal{H}_{\mu}(f)(z)\overline{g(z)}dA(z)\\
&=\frac{1}{\pi}\int_{\mathbb{U}}\int_{0}^{\infty}\frac{1}{t}f(\frac{z}{t})\, d\mu(t) \overline{g(z)}dA(z)\\
&=\frac{1}{\pi}\int_{0}^{\infty}\int_{\mathbb{U}}\frac{1}{t}f(\frac{z}{t}) \overline{g(z)}\,dA(z)\, d\mu(t)\\
&=\frac{1}{\pi}\int_{\mathbb{U}}f(z) \overline{\int_{0}^{\infty} t g(tz)\, d\mu(t)}\,dA(z),\\
\end{align*}
where we applied a change of variables and  Fubini's Theorem twice. This means that the adjoint  $\mathcal{H}_{\mu}^{*}$ of $\mathcal{H}_{\mu}$ on $A^{2}(\mathbb{U})$ is:

$$
\mathcal{H}_{\mu}^{*}(f)(z)=\int_{0}^{\infty} t f(tz)\, d\mu(t).
$$
We will consider $\mathcal{H}_{\mu}^{*}$ on $A^{p}(\mathbb{U})$ and suppose for a moment that it is well defined for functions in $A^{p}(\mathbb{U})$.
Let $\lambda(t)=t^{-1}, t>0$, then $\lambda$ maps $(0,\infty)$ onto $(0,\infty)$ and is measurable. Set $f(tz)=f_{z}(t)$ then
\begin{align*}
\mathcal{H}_{\mu}^{*}(f)(z)&=\int_{0}^{\infty} t f(tz)\, d\mu(t)\\
&=\int_{0}^{\infty} t f_{z}(t)\, d\mu(t)\\
&=\int_{0}^{\infty} \frac{1}{\lambda(t)} f_{z}\left(\frac{1}{\lambda(t)}\right)\, d\mu(t)\\
&=\int_{0}^{\infty} \frac{1}{t} f\left(\frac{z}{t}\right)\, d\nu(t),\\
&=\mathcal{H}_{\nu}(f)(z)
\end{align*}
where $d\nu =d \lambda_{*}(\mu)(t)$ and $\lambda_{*}(\mu)$ is the push-forward measure of $\mu$ with respect to $\lambda$.
We can now apply  the results of the first part of the paper to have:

\begin{Th}
Let $1\leq p< \infty$. The quasi-Hausdorff operator $\mathcal{H}_{\mu}^{*}$ is bounded on $A^{p}(\mathbb{U})$ if and only if
$$
\int_{0}^{\infty}t^{1-\frac{2}{p}}\,d\mu(t)<\infty.
$$
Moreover
$$
||\mathcal{H}_{\mu}^{*}||_{A^{p}(\mathbb{U})\to A^{p}(\mathbb{U})}=\int_{0}^{\infty}t^{1-\frac{2}{p}}\,d\mu(t).
$$
\end{Th}

\end{document}